\pdfoutput=1
\documentclass[final]{siamltex}

\usepackage{amsfonts}
\usepackage[leqno]{amsmath}
\usepackage{ amssymb }
\usepackage{tikz}
\usepackage{algorithmic}
\usepackage{eqparbox}
\usepackage{url}

\newtheorem{algorithm}{Algorithm}[section]

\newcommand{\vect}{\operatorname{vec}}

\newcommand{\ten}[1]{\mathcal{#1}}  

\title{A QR Algorithm for Symmetric Tensors}

\author{Kim Batselier \and Ngai Wong \thanks{Department of Electrical and Electronic Engineering, The University of Hong Kong, Hong Kong}}

\begin{document}

\maketitle

\begin{abstract}
We extend the celebrated QR algorithm for matrices to symmetric tensors. The algorithm, named QR algorithm for symmetric tensors (QRST), exhibits similar properties to its matrix version, and allows the derivation of a shifted implementation with faster convergence. We further show that multiple tensor eigenpairs can be found from a local permutation heuristic which is effectively a tensor similarity transform, resulting in the permuted version of QRST called PQRST. Examples demonstrate the remarkable effectiveness of the proposed schemes for finding stable and unstable eigenpairs not found by previous tensor power methods.
\end{abstract}

\begin{keywords}
Symmetric tensor, QR algorithm, shift, eigenpair
\end{keywords}

\begin{AMS}
15A69,15A18
\end{AMS}

\pagestyle{myheadings}
\thispagestyle{plain}
\markboth{KIM BATSELIER AND NGAI WONG}{QRST}

\section{Introduction}
\label{sec:intro}
The study of tensors has become intensive in recent years, e.g.,~\cite{tensorreview,LandsbergBook2012,de2000best,kofidis2002best,koldamayo11siam,koldamayo14arxiv}. Symmetric tensors arise naturally in various engineering problems. They are especially important in the problem of blind identification of under-determined mixtures \cite{Comon2006b,Ferreol2005,Lathauwer2007,lim2014blind}. Applications are found in areas such as speech, mobile communications, biomedical engineering and chemometrics. Although in principle, the tensor eigenproblem can be fully solved by casting it into a system of multivariate polynomial equations and solving them, this approach is only feasible for small problems and the scalability hurdle quickly defies its practicality. For larger problems, numerical iterative schemes are necessary and the higher-order power methods (HOPMs) presented in~\cite{de2000best,kofidis2002best,koldamayo11siam,koldamayo14arxiv} are pioneering attempts in this pursuit. Nonetheless, HOPM and its variants, e.g., SS-HOPM~\cite{koldamayo11siam} and GEAP~\cite{koldamayo14arxiv}, suffer from the need of setting initial conditions and that only one stable eigenpair is computed at a time.

In the context of matrix eigenpair computation, the QR algorithm is much more powerful and widely adopted, yet its counterpart in the tensor scenario seems to be lacking. Interestingly, the existence of a tensor QR algorithm was questioned at least a decade ago~\cite{kilmer2004decomposing}, which we try to answer here. Specifically, the central idea of the symmetric QR algorithm is to progressively align the original matrix (tensor) onto its eigenvector basis yet preserving the particular matrix (tensor) structure,  viz. symmetry, by similarity transform. In that regard, the tensor nature does not preclude itself from having a tensor QR algorithm counterpart. The main contribution of this paper is an algorithm, called the \textbf{QR} algorithm for \textbf{S}ymmetric \textbf{T}ensors (QRST), that computes the real eigenpair(s) of a real $d$th-order $n$-dimensional symmetric tensor $\ten{A}\in \mathbb{R}^{n\times\cdots\times n}$. We show that QRST also features a shifted version to accelerate convergence. Moreover, a permutation heuristic is introduced, captured in the algorithm named permuted QRST (PQRST), to scramble the tensor entries to produce possibly more distinct eigenpairs. We emphasize that the current focus is essentially a demonstration of the feasibility of QR algorithm for tensors, rather than an effort to develop efficient QRST implementations which will definitely involve further numerical work and smart shift design. We believe the latter will come about when QRST has become recognized.

This paper is organized as follows. Section~\ref{sec:notations} reviews some tensor notions and operations, and defines symmetric tensors and real eigenpairs. Section~\ref{sec:QRST} proposes QRST and PQRST, together with a shift strategy and convergence analysis. Section~\ref{sec:examples} demonstrates the efficacy of (P)QRST through numerical examples, and Section~\ref{sec:conclusions} draws the conclusion.

\section{Tensor Notations and Conventions}
\label{sec:notations}
Here we present some tensor basics and notational conventions adopted throughout this work. The reader is referred to~\cite{tensorreview} for a comprehensive overview. A general $d$th-order or $d$-way tensor, assumed real throughout this paper, is a multidimensional array $\ten{A}\in\mathbb{R}^{n_1\times n_2\times\cdots\times n_d}$ that can be perceived as an extension of the matrix format to its general $d$th-order counterpart~\cite{tensorreview}. We use calligraphic font (e.g., $\ten{A}$) to denote a general ($d>2$) tensor and normal font (e.g., $A$) for the specific case of a matrix ($d=2$). The inner product between two tensors $\ten{A}, \ten{B}\in\mathbb{R}^{n_1\times n_2\times\cdots\times n_d}$ is defined as%
\begin{align}
\label{eqn:inner}
\left<\ten{A},\ten{B}\right>&=\sum_{i_1,i_2,\cdots,i_d}\ten{A}_{i_1 i_2 \cdots i_d}\ten{B}_{i_1 i_2 \cdots i_d}\nonumber\\
&=\vect(\ten{A})^T \vect(\ten{B}),
\end{align}
where $\vect(\circ)$ is the vectorization operator that stretches the tensor entries into a tall column vector. Specifically, for the index tuple $[i_1 i_2\cdots i_d]$, the convention is to treat $i_1$ as the fastest changing index while $i_d$ the slowest, so $\vect(\ten{A})$ will arrange from top to bottom, namely, $\ten{A}_{11\cdots 1}$, $\ten{A}_{21\cdots 1}$, $\cdots$, $\ten{A}_{12\cdots 1}$, $\cdots$, $\ten{A}_{n_1n_2\cdots n_d}$. The norm of a tensor is often taken to be the Frobenius norm $||\ten{A}||=||\ten{A}||_F=<\ten{A},\ten{A}>^{1/2}$.

The $k$-mode product of a tensor $\ten{A}\in\mathbb{R}^{n_1\times\cdots \times n_k\times\cdots\times n_d}$ with a matrix $U\in\mathbb{R}^{p_k\times n_k}$ is defined by~\cite{tensorreview}%
\begin{align}
\label{eqn:kmode}
\left(\ten{A}{_{\times_k}} U\right)_{i_1\cdots i_{k-1} j_k i_{k+1}\cdots i_d}=\sum_{i_k=1}^{n_k}  U_{j_k i_k} \ten{A}_{i_1\cdots i_k\cdots i_d},%
\end{align}
and $\ten{A}{_{\times_k}} U\in\mathbb{R}^{n_1\times\cdots \times n_{k-1}\times p_k\times n_{k+1}\times\cdots\times n_d}$. For distinct modes in a series of tensor-matrix multiplication, the ordering of the multiplication is immaterial, namely,
\begin{align}
\label{eqn:kmodediff}
\ten{A}{_{\times_p}} B{_{\times_q}} C=\ten{A}{_{\times_q}} C{_{\times_p}} B\mbox{~~}(p\neq q),%
\end{align}
whereas it matters for the same mode,
\begin{align}
\label{eqn:kmodesame}
\ten{A}{_{\times_p}} B{_{\times_p}} C=\ten{A}{_{\times_p}} (CB).%
\end{align}
This is in fact quite easy to recognize by referring to Figure~\ref{fig:kmode} where $\ten{A}{_{\times_1}} B{_{\times_1}} C{_{\times_2}} D{_{\times_3}} E=\ten{A}{_{\times_1}} (CB) {_{\times_2}} D{_{\times_3}} E$. As an example, for matrices with compatible dimensions, it can be readily shown that $A{_{\times_1}} B{_{\times_2}} C=BAC^T$. Apparently, $\ten{A}{_{\times_k}} I_{n_k}=\ten{A}$ for $k=1,\cdots,d$ whereby $I_m$ denotes the $m\times m$ identity matrix.
\begin{figure}[t]
\begin{center}
	\includegraphics[width=.65\textwidth]{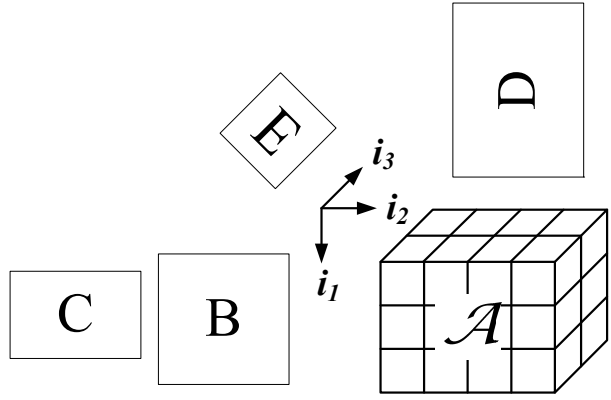}
	\end{center}
	\caption{Conceptual illustration of $\ten{A} {_{\times_1}} B {_{\times_1}} C {_{\times_2}} D {_{\times_3}} E$.}
	\label{fig:kmode}
\end{figure}

\subsection{Symmetric Tensors and Real Eigenpairs}
\label{subec:symeigen}
In this paper we focus on symmetric tensors. In particular, a $d$th-order symmetric tensor has all its dimensions being equal, namely, $n_1=\cdots =n_d=n$, and $\ten{A}_{i_1\ldots i_d} =  \ten{A}_{\pi(i_1\ldots i_d)}$, where $\pi(i_1\ldots i_d)$ is any permutation of the indices $i_1\ldots i_d$. To facilitate notation we further introduce some shorthand specific to symmetric tensors. First, we use $\mathbb{S}^{[d,n]}$ to denote the set of all real symmetric $d$th-order $n$-dimensional tensors. When a row vector $x^T\in \mathbb{R}^{1\times n}$ is to be multiplied onto $k$ modes of $\ten{A}\in \mathbb{S}^{[d,n]}$ (due to symmetry it can be assumed without loss of generality that all but the first $d-k$ modes are multiplied with $x^T$), we use the shorthand%
\begin{subequations}
\begin{align}
\ten{A}x & \triangleq \ten{A} {_{\times_d}} {x^T} \in\mathbb{S}^{[d-1,n]},\\
\ten{A}x^2 & \triangleq \ten{A} {_{\times_{d-1}}} {x^T} {_{\times_d}} {x^T} \in\mathbb{S}^{[d-2,n]},\cdots\\
\ten{A}x^{d-2} & \triangleq \ten{A} {_{\times_3}} {x^T} {_{\times_4}} {x^T} \cdots {_{\times_d}} x^T \in\mathbb{S}^{[2,n]}\in\mathbb{R}^{n\times n},\\
\ten{A}x^{d-1} & \triangleq \ten{A} {_{\times_2}} {x^T} {_{\times_3}} {x^T} \cdots {_{\times_d}}x^T \in\mathbb{S}^{[1,n]}=\mathbb{R}^n,\label{eqn:dbutonecontract}\\
\ten{A}x^{d} & \triangleq \ten{A} {_{\times_1}} {x^T} {_{\times_2}} {x^T} \cdots {_{\times_d}}x^T \in\mathbb{S}^{[0,n]}=\mathbb{R}.%
\end{align}
\end{subequations}
This is a \emph{contraction} process because the resulting tensors are progressively shrunk wherein we assume any higher order singleton dimensions ($n_i=1$) are ``squeezed'', e.g., $\ten{A}x^{d-1}$ is treated as an $n\times n$ matrix rather than its equivalent $n\times n\times 1 \cdots \times 1$ tensor. Besides contraction by a row vector, the notation generalizes to any matrix of appropriate dimensions. Suppose we have a matrix $V\in \mathbb{R}^{n\times p} (p>1)$, then%
\begin{align}
\ten{A}V^{d-1}\triangleq \ten{A} {_{\times 2}} {V^T} {_{\times 3}} {V^T} \cdots {_{\times d}} {V^T}\in \mathbb{R}^{n\times p\cdots \times p}.%
\end{align}

Now, the real eigenpair of a symmetric tensor $\ten{A}\in\mathbb{S}^{[d,n]}$, namely, $(\lambda,x)\in \mathbb{R}\times\mathbb{R}^n$, is defined as
\begin{align}
\ten{A}x^{d-1}=\lambda x,\mbox{~~}x^Tx=1.
\end{align}
This definition is called the (real) Z-eigenpair~\cite{qi05jsc} or $l^2$-eigenpair~\cite{Lim_singularvalues} which is of interest in this paper, apart from other differently defined eigenpairs, e.g.,~\cite{koldamayo14arxiv}. It follows that%
\begin{align}
x^T \ten{A}x^{d-1}=\ten{A}x^{d}=\lambda.
\end{align}
For $\ten{A}\in \mathbb{S}^{[d,n]}$, depending on whether $d$ is odd or even, we have the following result regarding eigenpairs that occur in pairs~\cite{koldamayo11siam}:
\begin{lemma}
For an even-order symmetric tensor, if $(\lambda, x)$ is an eigenpair, then so is $(\lambda, -x)$. For an odd-order symmetric tensor, if $(\lambda, x)$ is an eigenpair, then so is $(-\lambda, -x)$.
\label{lemma:eigenpairs}
\end{lemma}
\begin{proof}
For an even $d$, we have $\ten{A}(-x)^{d-1}=-\ten{A}x^{d-1}=-\lambda x=\lambda (-x)$. For an odd $d$, we have $\ten{A}(-x)^{d-1}=\ten{A}x^{d-1}=\lambda x=-\lambda (-x)$.%
\end{proof}

Subsequently, we do not call two eigenpairs distinct if they follow the relationship in Lemma~\ref{lemma:eigenpairs}. For even-order symmetric tensors, there exists a symmetric \emph{identity tensor} counterpart $\ten{E}\in\mathbb{S}^{[d,n]}$ such that~\cite{qi05jsc,koldamayo11siam}%
\begin{align}
\ten{E}x^{d-1}=x~\mbox{~for all~}~x^Tx=1.
\label{eqn:identity}
\end{align}
The constraint that $||x||=1$ is necessary so it holds for all even $d\ge 2$. Obviously, $(1,x)$ is an eigenpair for $\ten{E}$. It is also obvious that such an identity tensor does not exist for an odd $d$ since contracting with $-x$ does not lead to itself. Consequently, for an even $d$, we have similar shifting of eigenvalues when we add a multiple of $\ten{E}$ to $\ten{A}$, i.e., if $(\lambda,x)$ is an eigenpair of $\ten{A}$, then $(\lambda+\alpha,x)$ is an eigenpair of $\ten{A}+\alpha \ten{E}$.

\subsection{Similar Tensors}
\label{subec:similar}
In this paper, we say that two symmetric tensors $\ten{A}, \ten{B}\in \mathbb{S}^{[d,n]}$ are \emph{similar}, denoted by $\ten{A} \sim \ten{B}$, if for some non-singular $P\in \mathbb{R}^{n\times n}$%
\begin{align}
\ten{B}=\ten{A}P^{d},
\end{align}
so that the \emph{inverse transform} of $\ten{B}$ by $P^{-1}$ will return $\ten{A}$, namely,
\begin{align}
\ten{B}P^{-d}=\left(\ten{A}P^{d}\right)P^{-d}=\ten{A} {_{\times 1}} (P^{-T}P^T) {_{\times 2}} (P^{-T}P^T) \cdots {_{\times d}} (P^{-T}P^T)=\ten{A}.
\end{align}
Suppose $\ten{\tilde{A}}\in\mathbb{S}^{[d,n]}$ is similar to $\ten{A}$ through a similarity transform by an orthogonal matrix $\underline{Q}$, i.e., $\underline{Q}^T \underline{Q}=I_n$ and $\ten{\tilde{A}}=\ten{A}\underline{Q}^d$. It follows that if $(\lambda, e_1)$ is an eigenpair of $\ten{\tilde{A}}$, then $(\lambda, \underline{q}_1)$ is an eigenpair of $\ten{A}$ where $\underline{q}_1$ denotes the first column of $\underline{Q}$. This is seen by%
\begin{align}
\ten{\tilde{A}}e_1^{d-1}=\lambda e_1\Leftrightarrow & \ten{A} {_{\times 1}}\underline{Q}^T {_{\times 2}} (e_1^T \underline{Q}^T) \cdots {_{\times d}}(e_1^T \underline{Q}^T)=\lambda e_1\nonumber\\
\Leftrightarrow & \ten{A} {_{\times 1}}\underline{Q}^T {_{\times 2}}(\underline{q}_1)^T \cdots {_{\times d}}(\underline{q}_1)^T =\lambda e_1\nonumber\\
\Leftrightarrow & \underline{Q}^T\ten{A}\underline{q}_1^{d-1}=\lambda e_1\nonumber\\
\Leftrightarrow & \ten{A}\underline{q}_1^{d-1}=\lambda \underline{q}_1.%
\label{eqn:e1q1}
\end{align}
In fact, we do not need to limit ourselves to $e_1$ and the above result easily generalizes to any $e_i$, which we present as lemma~\ref{lemma:eiqi}. 
\begin{lemma}
For a symmetric tensor $\ten{A}\in \mathbb{S}^{[d,n]}$, if $(\lambda, e_i)$ is an eigenpair of $\ten{A}\underline{Q}^d$, then $(\lambda, \underline{q}_i)$ is an eigenpair of $\ten{A}$ where $\underline{q}_i$ is the $i$th column of the orthogonal matrix $\underline{Q} \in \mathbb{R}^{n\times n}$.
\label{lemma:eiqi}
\end{lemma}
\begin{proof}
The proof easily follows from~(\ref{eqn:e1q1}).
\end{proof}

\section{QR Algorithm for Symmetric Tensors}
\label{sec:QRST}
The conventional QR algorithm for a symmetric matrix starts with an initial symmetric matrix $A_0$, and computes its QR factorization $Q_1R_1=A_0$. Next, the QR product is reversed to give $A_1=R_1Q_1=Q_1^TA_0 Q_1$. And in subsequent iterations, we have $Q_k R_k=A_{k-1}$, and $A_k=Q_k^TA_{k-1}Q_k=(Q_1\cdots Q_k)^T A_0 (Q_1\cdots Q_k)=\underline{Q}_k^T A_0 \underline{Q}_k$ where we define $\underline{Q}_k=Q_1Q_2\cdots Q_k$. Under some mild assumptions, $A_k$ will converge to a diagonal matrix holding the eigenvalues of $A_0$. In the following, we present QRST which computes multiple, stable and unstable~\cite{koldamayo11siam}, eigenpairs. Moreover, with a permutation strategy to scramble tensor entries, PQRST can produce possibly even more distinct eigenpairs, rendering such tensor QR algorithm much more effective compared to tensor power methods~\cite{koldamayo11siam,koldamayo14arxiv} which find only one stable eigenpair at a time. 

\subsection{QRST}
\label{subsec:QRST}
\begin{figure}[t]
\begin{center}
	\includegraphics[width=4.5in]{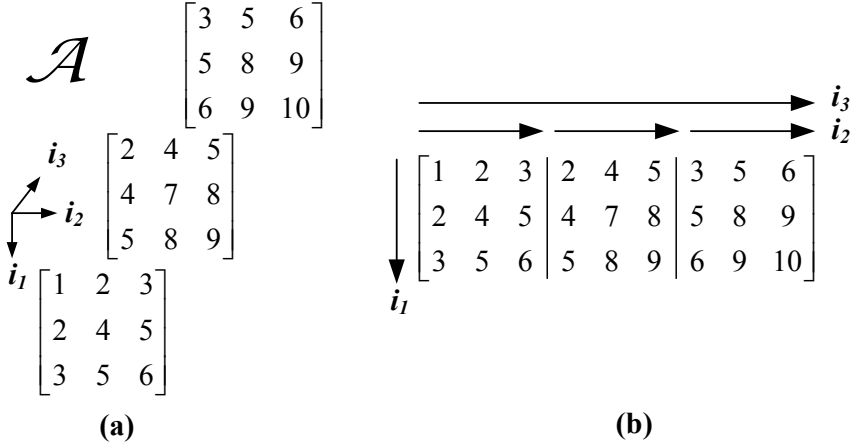}
	\end{center}
	\caption{(a) The example labeling tensor in $\mathbb{S}^{[3,3]}$ with sequential numbering of each unique entry. (b) Matricization about the first mode.}
	\label{fig:wongd3n3}
\end{figure}
The symmetric tensor $\ten{A}\in\mathbb{S}^{[3,3]}$ in Figure~\ref{fig:wongd3n3}(a) will be used as a running example. We call it a labeling matrix as it labels each distinct entry in a symmetric tensor sequentially with a natural number. Moreover, Figure~\ref{fig:wongd3n3}(b) is the \emph{matricization} of this labeling matrix along the first mode~\cite{tensorreview}. As a matter of fact, matricization about any mode of $\ten{A}$ gives rise to the same matrix due to symmetry. 

Now we formally propose the QRST algorithm, summarized in Algorithm~\ref{alg:qrst}. The core of QRST is a two-step iteration, remarked in Algorithm~\ref{alg:qrst} as Steps 1 and 2, for the $i$th slice where $i=1,\cdots,n$. In the Algorithm, it should be noted that Matlab convention is used whereby it can be recognized that $\ten{A}_{k}(:,i,i,\cdots,i)=\ten{A}_{k}e_i^{d-1}$ is a column and $\ten{A}_{k}(:,:,i,\cdots,i)=\ten{A}_{k}e_i^{d-2}$ is a (symmetric) square matrix.
\\
\\
\framebox[1.05\textwidth][l]{\begin{minipage}{\textwidth}
\begin{algorithm}{(Shifted) QRST}\\
\textit{\textbf{Input}}: {$\ten{A}_0\in\mathbb{S}^{[d,n]}$, tolerance $\tau$, maximum iteration $k_{max}$, shifts $s_k$ from shift scheme}\\
\textit{\textbf{Output}}: {eigenpairs $(\lambda_i,\underline{q}_i)$}
\begin{algorithmic}
\FOR[iterate through the $i$th square slice $\ten{A}_0(:,:,i,\cdots,i)$]{$i=1:n$}
\STATE{$\epsilon \gets ||\ten{A}_{0}e_i^{d-1}-e_i||_2 / ||\ten{A}_0(:,:,i,\cdots,i)||_2$}
\STATE{$\underline{Q}\gets I_n$} 
\STATE{iteration count $k \gets 1$}
\WHILE{$\epsilon >\tau$ or $k\le k_{max}$}
\STATE{Step 1: $Q_k R_k=\ten{A}_{k-1}(:,:,i,\cdots,i)+s_{k-1} I_n$~~~{\#~{\tt perform QR factorization}}}
\STATE{Step 2: $\ten{A}_k=\ten{A}_{k-1}Q_k^d$~~~{\#~{\tt similarity transform}}}
\STATE{$\underline{Q}\gets \underline{Q}Q_k$}
\STATE{$\epsilon \gets ||\ten{A}_{k}e_i^{d-1}-e_i||_2 / ||\ten{A}_k(:,:,i,\cdots,i)||_2$}
\ENDWHILE
\ENDFOR
\IF{converged}
\STATE{$\lambda_i \gets \ten{A}_k(i,i,i,\cdots,i)$}
\STATE{$\underline{q}_i\gets \underline{Q}(:,i)$}
\ENDIF
\end{algorithmic}
\label{alg:qrst}%
\end{algorithm}
\end{minipage}}
\\
\\
Referring to Figure~\ref{fig:wongd3n3}(b), if we take $\ten{A}_0=\ten{A}$, $i=1$ and start with $\ten{A}_0 e_1^{d-2}=\ten{A}_0 e_1$, the first (leftmost) $3\times 3$ ``slice'' of the $3\times 9$ matrix is extracted for QR factorization to get $Q_1 R_1$, and then all three modes of $\ten{A}_0$ are multiplied with $Q_1^T$ so the tensor symmetry is preserved. Using Algorithm~\ref{alg:qrst} with the heuristic shift described in Section~\ref{subsec:convergence}, this tensor QR algorithm converges in $152$ iterations when $\ten{A}_{k}e_1^{2}$ becomes a multiple of $e_1$ (i.e., all but the first entry in the $\ten{A}_{k}e_1^{2}$ column are numerically zero), resulting in%
\begin{align*}
\left[ {\begin{array}{*{20}c}
   -0.14 & 0.00 & 0.00 &\vline &  0.00 & -0.19 & 0.56 &\vline &  0.00 & 0.56 & -0.40  \\
   0.00 & -0.19 & 0.56 &\vline &  -0.19 & -20.69 & 11.56 &\vline &  0.56 & 11.56 & -5.56  \\
   0.00 & 0.56 & -0.40 &\vline &  0.56 & 11.56 & -5.56 &\vline &  -0.40 & -5.56 & 2.46  \\
 \end{array} } \right],%
\end{align*}
for which $e_1=[1~0~0]^T$ is obviously an eigenvector as its contraction along the $i_2$ and $i_3$ axes will extract the first column $[-0.14~0~0]^T$. Using~(\ref{eqn:e1q1}), this corresponds to the eigenpair $(-0.14,[-0.7854~0.6029~-0.1401]^T)$ of the labeling tensor. In fact, for different $i$'s, we shoot for different converged patterns in the corresponding square slices as depicted in Figure~\ref{fig:d3n3converged}, where it is obvious that $\ten{A}_ke_i^2=\lambda_i e_i$ for each $i$.%
\begin{figure}[t]
\begin{center}
	\includegraphics[width=.9\textwidth]{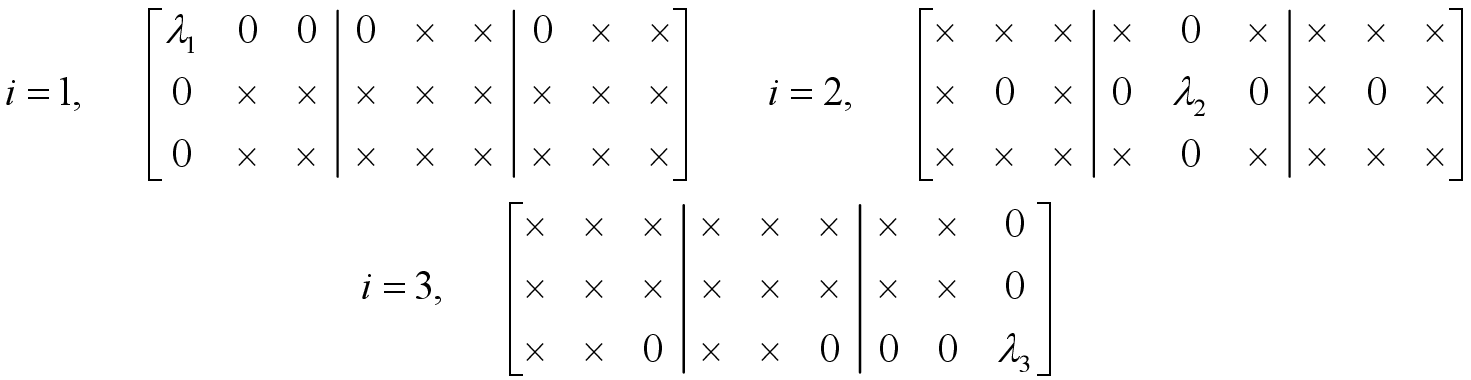}
	\end{center}
	\caption{The ``signatures'' of convergence in the example labeling tensor for $i=1,2,3$ where $\times$ denotes ``don't care''. Positions of $0$'s are obvious by comparing with Figure~\ref{fig:wongd3n3}(b).}
	\label{fig:d3n3converged}
\end{figure}

Since the core computation of QRST lies in the QR factorization ($\mathcal{O}(n^3)$) and the corresponding similarity transform ($\mathcal{O}(d n^2 n^{d-1})=\mathcal{O}(dn^{d+1})$), the overall complexity is therefore dominated by the similarity transform process. Note that this high complexity comes about whereby the tensor symmetry is not exploited. When a low-rank decomposition for $\ten{A}_0$ exists that represents $\ten{A}_0$ as a finite sum of $r$ rank-$1$ outer factors, the complexity of QRST can be largely reduced to $\mathcal{O}(drn^2)$. However, we refrain from elaborating as it is beyond our scope of verifying the idea of QRST. Such low-rank implementation will be presented separately in a sequel of this work.

\subsection{Convergence and Shift Selection}
\label{subsec:convergence}
We assume an initial symmetric $\ten{A}_0=\ten{A}\in\mathbb{S}^{[d,n]}$. The vector $e_i$ denotes the $i$th column of the identity matrix $I_n$ and we define $\underline{Q}_k\triangleq Q_1Q_2\cdots Q_k$ with $\underline{Q}_0=I_n$ and $\underline{Q}_1=Q_1$ etc. A tensor with higher order singleton dimensions is assumed squeezed, e.g., $\ten{A}_0 e_i^{d-2}$ is regarded as an $n\times n$ matrix rather than its equivalent $n\times n\times 1\cdots \times 1$ tensor. Also, we use $\underline{q}_{i,k}$ to denote the $i$th column in $\underline{Q}_k$ and $r_{ij,k}$ to stand for the $ij$-entry of $R_k$. It will be shown that the unshifted QRST operating on the matrix slice resulting from contracting all but the first two modes $\ten{A}_k$ by $e_i$, i.e., $\ten{A}_k e_i^{d-2}$, is in fact an orthogonal iteration. This can be better visualized by enumerating a few iterations of QRST:
\begin{subequations}
\begin{align}
Q_1 R_1 &= \ten{A}_0 e_i^{d-2},\\
\ten{A}_1 &=\ten{A}_0 Q_1^d=\ten{A}_0 \underline{Q}_1^d,\label{eqn:qrstex_1}\\
Q_2 R_2 &= \ten{A}_1 e_i^{d-2} = \ten{A}_0 {_{\times_1}} {\underline{Q}_1^T} {_{\times_2}} {\underline{Q}_1^T} {_{\times_3}} {\underline{q}_{i,1}^T}\cdots {_{\times_d}} {\underline{q}_{i,1}^T} =  \underline{Q}_1^T (\ten{A}_0 \underline{q}_{i,1}^{d-2}) \underline{Q}_1,\\
\ten{A}_2 &=\ten{A}_1 Q_2^d=(\ten{A}_0 \underline{Q}_1^d) Q_2^d=\ten{A}_0 (\underline{Q}_1 Q_2)^d=\ten{A}_0 \underline{Q}_2^d,\label{eqn:qrstex_2}\\
Q_3 R_3 &= \ten{A}_2 e_i^{d-2} = \ten{A}_0 {_{\times_1}} {\underline{Q}_2^T} {_{\times_2}} {\underline{Q}_2^T} {_{\times_3}} {\underline{q}_{i,2}^T}\cdots {_{\times_d}} {\underline{q}_{i,2}^T} =  \underline{Q}_2^T (\ten{A}_0 \underline{q}_{i,2}^{d-2}) \underline{Q}_2,\label{eqn:qrstex_3}%
\end{align}
\label{eqn:qrstex}%
\end{subequations}
etc. From~(\ref{eqn:qrstex}), we see that in general for $k=1,2,\cdots$,%
\begin{subequations}
\begin{align}
\ten{A}_{k-1} &=\ten{A}_0 \underline{Q}_{k-1}^{d},\label{eqn:orthoIter1}\\
Q_kR_k &= {\ten{A}_{k-1}} {e_i^{d-2}} = \underline{Q}_{k-1}^T (\ten{A}_0 \underline{q}_{i,k-1}^{d-2}) \underline{Q}_{k-1}.\label{eqn:orthoIter2}%
\end{align}
\end{subequations}
Multiplying $\underline{Q}_{k-1}$ to both sides of~(\ref{eqn:orthoIter2}) we get%
\begin{align}
&\underline{Q}_k R_k =  (\ten{A}_0 \underline{q}_{i,k-1}^{d-2}) \underline{Q}_{k-1} \Leftrightarrow\nonumber\\
&\left[ {\begin{array}{*{20}c}
   | & | & {} & |  \\
   {\underline{q}_{1,k} } & {\underline{q}_{2,k} } &  \cdots  & {\underline{q}_{n,k} }  \\
   | & | & {} & |  \\
 \end{array} } \right]\left[ {\begin{array}{*{20}c}
   {r_{11,k} } & {r_{12,k} } &  \cdots  & {r_{1n,k} }  \\
   {} & {r_{22,k} } & {} & {r_{2n,k} }  \\
   {} & {} &  \ddots  &  \vdots   \\
   {} & {} & {} & {r_{nn,k} }  \\
 \end{array} } \right] \nonumber\\
&~~~~~~~~~~~~~~= \left[ {\begin{array}{*{20}c}
   {} & {} & {}  \\
   {} & {\ten{A}_0\underline{q}_{i,k-1}^{d-2} } & {}  \\
   {} & {} & {}  \\
 \end{array} } \right]\left[ {\begin{array}{*{20}c}
   | & | & {} & |  \\
   {\underline{q}_{1,k - 1} } & {\underline{q}_{2,k - 1} } &  \cdots  & {\underline{q}_{n,k - 1} }  \\
   | & | & {} & |  \\
 \end{array} } \right].\label{eqn:qrstcolumns}%
\end{align}
It should now become obvious that~(\ref{eqn:qrstcolumns}) is the tensor generalization of the matrix orthogonal iteration~\cite{matrixcomputations} (also called simultaneous iteration or subspace iteration) for finding dominant invariant eigenspace of matrices\footnote{Unlike the matrix case with a constant (static) matrix $A$, the tensor counterpart comes with a \emph{dynamic} $\ten{A}_0\underline{q}_{i,k-1}^{d-2}$ that is updated in every pass with the $i$th column of $\underline{Q}_{k-1}$.}. This result comes at no surprise as the matrix QR algorithm is known to be equivalent to orthogonal iteration applied to a block column vector. However, our work confirms, for the first time, this is also true also for symmetric tensors.

Now, if we set $i=1$ and extract the first columns on both sides of~(\ref{eqn:qrstcolumns}), we get%
\begin{align}
\underline{q}_{1,k}r_{11,k} =(\ten{A}_0\underline{q}_{1,k-1}^{d-2})\underline{q}_{1,k-1}=\ten{A}_0 \underline{q}_{1,k}^{d-1}
\mbox{~~or~~}\underline{q}_{1,k} =\frac{\ten{A}_0 \underline{q}_{1,k}^{d-1}}{||\ten{A}_0 \underline{q}_{1,k}^{d-1}||},\label{eqn:1stcolumn}%
\end{align}
which is just the unshifted S-HOPM method~\cite{de2000best,kofidis2002best} cited as Algorithm~1 in~\cite{koldamayo11siam}. If we start QRST with $i=2$, then the first two columns of~(\ref{eqn:qrstcolumns}) proceed as%
\begin{subequations}
\begin{align}
\underline{q}_{1,k} r_{11,k}&=(\ten{A}_0\underline{q}_{2,k-1}^{d-2})\underline{q}_{1,k-1},\label{eqn:2ndcolumn1}\\
\underline{q}_{1,k}r_{12,k}+\underline{q}_{2,k}r_{22,k}& =\ten{A}_0\underline{q}_{2,k-1}^{d-1}.\label{eqn:2ndcolumn2}
\end{align}
\label{eqn:2ndcolumn}
\end{subequations}
Assuming convergence at $k=\infty$,%
\begin{subequations}
\begin{align}
\underline{q}_{1,\infty} r_{11,\infty}&=(\ten{A}_0\underline{q}_{2,\infty}^{d-2})\underline{q}_{1,\infty},\label{eqn:2ndcolumn1converged}\\
\underline{q}_{1,\infty}r_{12,\infty}+\underline{q}_{2,\infty}r_{22,\infty}& =(\ten{A}_0\underline{q}_{2,\infty}^{d-2})\underline{q}_{2,\infty}=\ten{A}_0\underline{q}_{2,\infty}^{d-1},\label{eqn:2ndcolumn2converged}%
\end{align}
\label{eqn:2ndcolumnconverged}%
\end{subequations}
from which we can easily check that $r_{12,\infty}=\underline{q}_{1,\infty}^T (\ten{A}_0\underline{q}_{2,\infty}^{d-2})\underline{q}_{2,\infty}=r_{11,\infty} \underline{q}_{1,\infty}^T \underline{q}_{2,\infty}=0$ due to~(\ref{eqn:2ndcolumn1converged}). In fact, for any $i$, if $\underline{Q}_k$ and $R_k$ converge to $\underline{Q}_{\infty}$ and $R_{\infty}$, respectively, then from~(\ref{eqn:qrstcolumns}), we must have
\begin{align}
R_{\infty}=\left[ {\begin{array}{*{20}c}
   {r_{11,\infty} } & 0 &  \cdots  & 0  \\
   0 & {r_{22,\infty} } & {} & \vdots  \\
   {\vdots} & {} &  \ddots  &  0   \\
   0 & {\cdots} & {0} & {r_{nn,\infty} }  \\
 \end{array} } \right]=\underline{Q}_{\infty}^T (\ten{A}_0\underline{q}_{i,\infty}^{d-2}) \underline{Q}_{\infty},\label{eqn:Rdiagonal}%
\end{align}
where $R_{\infty}$ is diagonal due to symmetry on the right of equality. Now if we pre-multiply $\underline{Q}_{\infty}$ onto both sides of~(\ref{eqn:Rdiagonal}) and post-multiply with $e_j$, then %
\begin{align}
\underline{q}_{j,\infty}r_{jj,\infty}=(\ten{A}_0\underline{q}_{i,\infty}^{d-2})\underline{q}_{j,\infty}\label{eqn:allconverged}
\end{align}
which boils down to the eigenpair $(r_{ii,\infty},\underline{q}_{i,\infty})$ when $i=j$. So in effect, (shifted) QRST finds more than just the conventional eigenpairs, but actually all ``local'' eigenpairs of the square matrix $(\ten{A}_0\underline{q}_{i,\infty}^{d-2})$, one of them happens to be the tensor eigenpair. This explains why in practice, (shifted) QRST can locate even unstable tensor eigenpairs as it is not restricted to the ``global'' convexity/concavity constraint imposed on the function $f(x)=\ten{A}_0x^d$ that always leads to positive/negative stable eigenpairs~\cite{koldamayo11siam}.

Till now, it can be seen that except for $i=1$, (shifted) QRST differs from the (S)S-HOPM algorithm presented in~\cite{koldamayo11siam,koldamayo14arxiv}, and convergence is generally NOT ensured for $i\neq 1$. However, inspired by~(\ref{eqn:2ndcolumn2converged}), the key for convergence starts with the convergence of the first column that reads%
\begin{align}
\underline{q}_{1,k} &=\frac{(\ten{A}_0 \underline{q}_{i,k-1}^{d-2})\underline{q}_{1,k-1}}{||(\ten{A}_0 \underline{q}_{i,k-1}^{d-2})\underline{q}_{1,k-1}||}.\label{eqn:1stcolumnconverged}
\end{align}
Consequently, applying results in~\cite{koldamayo11siam,koldamayo14arxiv}, convergence of~(\ref{eqn:1stcolumnconverged}) can be enforced locally around $\underline{q}_{1,k-1}$ provided the scalar function $f(x)=x^T(\ten{A}_0 \underline{q}_{i,k-1}^{d-2})x$ is convex (concave), which is ``localized'' in the sense that the contracted $(\ten{A}_0 \underline{q}_{i,k-1}^{d-2})$ is used instead of the whole $\ten{A}_0$. Then, for example, convexity can be achieved by adding a shift $s_{k-1}$ that makes $(\ten{A}_0 \underline{q}_{i,k-1}^{d-2})+s_{k-1} I_n$ positive semidefinite, namely,
\begin{align}
s_{k-1}& \ge -\lambda_{\min}\left(\ten{A}_{0} \underline{q}_{i,k-1}^{d-2}\right)\nonumber\\
& = -\lambda_{\min}\left(\underline{Q}_{k-1}^T(\ten{A}_{0} \underline{q}_{i,k-1}^{d-2})\underline{Q}_{k-1}\right) \nonumber\\
& = -\lambda_{\min}\left(\ten{A}_{k-1}e_i^{d-2}\right).\label{eqn:shiftedqrst}%
\end{align}
So the shift step in Algorithm~\ref{alg:qrst} is simply $s_{k-1}=-\lambda_{\min}(\ten{A}_{k-1}(:,:,i,\cdots,i))+\delta$ where $\delta$ is a small positive quantity which is set to be unity in our tests unless otherwise specified. This is analogous to the core step in SS-HOPM (Algorithm 2 of~\cite{koldamayo11siam}). In short, the shift strategy in~\cite{koldamayo11siam,koldamayo14arxiv} can be locally reused for QRST with the advantage that QRST may further compute the unstable eigenpairs alongside the positive/negative-stable ones, since a group of eigenpairs are generated at a time due to analogy of QRST to orthogonal iteration.

\subsection{PQRST}
\label{subsec:PQRST}
For a symmetric $\ten{A}_0\in \mathbb{S}^{[d,n]}$, QRST produces at most $n$ eigenpairs in one pass, and fewer than $n$ eigenpairs can result when QRST does not converge or when the same eigenpair is computed repeatedly. To increase the probability of generating distinct eigenpairs, and noting that eigenvectors of similar tensors are just a change of basis, we introduce the heuristic of permutation to scramble the entries in $\ten{A}_0$ as much as possible. Figure~\ref{fig:wongd3n3p} shows how the $\mathbb{S}^{[3,3]}$ labeling tensor is permuted by the $n!$ $(n=3)$ (orthogonal) permutation matrices.

\begin{figure}[ht]
\begin{center}
	\includegraphics[width=\textwidth]{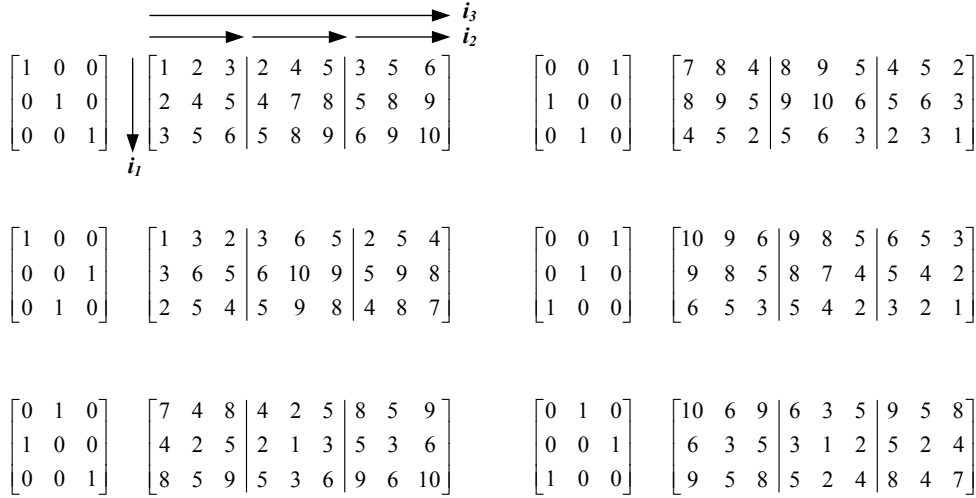}
	\end{center}
	\caption{Successive permutations of the example labeling tensor. Here the corresponding permutation matrices are put alongside the flattened tensors.}
	\label{fig:wongd3n3p}
\end{figure}
As such, we can run QRST on a tensor subject to different permutations. We call this permuted QRST or PQRST, summarized in Algorithm~\ref{alg:pqrst}.
\\
\\
\framebox[1.05\textwidth][l]{\begin{minipage}{\textwidth}
\begin{algorithm}{PQRST}\\
\textit{\textbf{Input}}: {$\ten{A}_0\in\mathbb{S}^{[d,n]}$, preconditioning orthogonal matrices in a cell $P\{p\}$, $p=1,\cdots,\bar{p}$}\\
\textit{\textbf{Output}}: {eigenpairs $(\lambda_i,\underline{q}_i)$}
\begin{algorithmic}
\FOR{$p=1:\bar{p}$}
\STATE{$P\gets P\{p\}$}
\STATE{$\ten{A}_p\gets \ten{A}_0 P^d$}
\STATE{call QRST Algorithm~\ref{alg:qrst} with $\ten{A}_p$ and collect distinct converged eigenpairs} 
\ENDFOR
\end{algorithmic}
\label{alg:pqrst}%
\end{algorithm}
\end{minipage}}\\
\\
Of course, PQRST requires running QRST $n!$ times which quickly becomes impractical for large $n$. In that case, we should limit the number of permutations by sampling across all possible permutations or use other heuristics for setting the initial pre-transformed $\ten{A}_0$, which is our on-going research.

\section{Numerical Examples}
\label{sec:examples}
In this section, we compare PQRST with SS-HOPM from \cite{koldamayo11siam}, and with its adaptive version described in \cite{koldamayo14asshopm}. For all the SS-HOPM runs the shift was chosen as $(m-1)\sum_{i_1,\ldots,i_m} |a_{i_1\cdots i_m}|$. We considered the iteration to have converged when $|\lambda_{k+1}-\lambda_k| \leq 10^{-15}$. If the eigenpairs are not known in advance then we also solve the corresponding polynomial system using the numerical method described in \cite{PhilippeDreesenThesis}. All numerical experiments were done on a quad-core desktop computer with 16 GB memory. Each core runs at 3.3 GHz.

\subsection{Example 1}
\label{subsec:Ex1}
As a first example we use the labeling tensor of Figure \ref{fig:wongd3n3}(a). Due to the odd order of the tensor we have that both $(\lambda,x)$ and $(-\lambda,-x)$ are eigenpairs. From solving the polynomial system we learn that there are 4 distinct eigenpairs, listed in Table \ref{tbl:roots1}. The stability of each of these eigenpairs was determined by checking the signs of the eigenvalues of the projected Hessian of the Lagrangian as described in \cite{koldamayo11siam}.
\begin{table}[t!]
\begin{center}
\caption{Distinct eigenpairs for labeling tensor of Figure \ref{fig:wongd3n3}(a). }   
\begin{tabular}{|c|c|c|}            
\hline                        
$\lambda$ & $x$ &  stability\\ [0.5ex]   
\hline                              
30.4557 & $\begin{pmatrix}0.37 & 0.61 & 0.70 \end{pmatrix} $ & negatively stable \\              
\hline 
0.4961 & $\begin{pmatrix}-0.80 & -0.34 & 0.50 \end{pmatrix}$ & positively stable\\               
\hline 
0.1688 & $\begin{pmatrix}0.86 & -0.44 & -0.23 \end{pmatrix}$ & positively stable\\              
\hline                              
0.1401 & $\begin{pmatrix}0.78 & -0.60 & 0.14 \end{pmatrix}$ & unstable\\              
\hline 
\end{tabular}
\label{tbl:roots1}          
\end{center}
\end{table}

100 runs of the SS-HOPM algorithm were performed, each with a different random initial vector. Choosing the random initial vectors from either a uniform or normal distribution had no effect whatsoever on the final result. The shift was fixed to the ``conservative" choice of $(m-1)\sum_{i_1,\ldots,i_m} |a_{i_1\cdots i_m}|=288$. All 100 runs converged to the eigenpair with the eigenvalue $30.4557$.  The median of number of iterations is 169. One run of 169 iterations took about 0.02 seconds. Next, 100 runs of the adaptive SS-HOPM algorithm were performed with a tolerance of $10^{-15}$. When the initial guess for $x$ was a random vector sampled from a uniform distribution, then all 100 runs consistently converged to the $30.4557$ eigenpair with a median number of 14 iterations. This clearly demonstrates the power of the adaptive shift in order to reduce the number of iterations. Choosing the initial vector $x$ from a normal distribution resulted in finding the 3 stable eigenpairs and 2 iterations not converging (Table \ref{tbl:asshopm1}). The error was computed as the average value of $||\ten{A}x^{m-1} - \lambda\,x ||_2$.
Similar to the SS-HOPM algorithm, the PQRST algorithm without shifting can only find the ``largest" eigenpair with a median of 13 iterations. This took about 0.082 seconds. PQRST with shift $\delta=1$ results in 18 runs that find all eigenpairs (Table \ref{tbl:pqrst1}). From the 18 runs, 8 did not converge.

\begin{table}[t!]
\begin{center}
\caption{Adaptive SS-HOPM for the labeling tensor of Figure \ref{fig:wongd3n3}(a). }   
\begin{tabular}{|c|c|c|c|c|}            
\hline                        
occ. & $\lambda$ & $x$ & median its & error \\ [0.5ex]   
\hline                              
58 & 30.4557 & $\begin{pmatrix}0.37 & 0.61 & 0.70 \end{pmatrix} $ & 15 & $4.95\times 10^{-15}$\\               
\hline 
23 & 0.4961 & $\begin{pmatrix}-0.80 & -0.34 & 0.50 \end{pmatrix}$  & 225 & $1.09\times 10^{-14}$\\                  
\hline 
17 & 0.1688 & $\begin{pmatrix}0.86 & -0.44 & -0.23 \end{pmatrix}$  & 701 & $4.85\times 10^{-14}$\\                  
\hline                              
\end{tabular}
\label{tbl:asshopm1}          
\end{center}
\end{table}

\begin{table}[t!]
\begin{center}
\caption{PQRST with shift for the labeling tensor of Figure \ref{fig:wongd3n3}(a). }   
\begin{tabular}{|c|c|c|c|c|}            
\hline                        
occ. & $\lambda$ & $x$ & median its & error \\ [0.5ex]   
\hline                              
2 & 30.4557 & $\begin{pmatrix}0.37 & 0.61 & 0.70 \end{pmatrix} $ & 18 & $3.10\times 10^{-14}$\\               
\hline 
2 & 0.4961 & $\begin{pmatrix}-0.80 & -0.34 & 0.50 \end{pmatrix}$  & 69.5 & $6.96\times 10^{-15}$\\                  
\hline 
2 & 0.1688 & $\begin{pmatrix}0.86 & -0.44 & -0.23 \end{pmatrix}$  & 246 & $4.74\times 10^{-15}$\\                  
\hline                              
4 & 0.1401 & $\begin{pmatrix}0.78 & -0.60 & 0.14 \end{pmatrix}$  & 201 & $4.24\times 10^{-15}$\\                  
\hline 
\end{tabular}
\label{tbl:pqrst1}          
\end{center}
\end{table}

\subsection{Example 2}
\label{subsec:Ex2}
For the second example we consider the 4th-order tensor from Example 3.5 in \cite{koldamayo11siam}. Solving the polynomial system shows that there are 11 distinct eigenpairs (Table \ref{tbl:roots2}).
\begin{table}[t!]
\begin{center}
\caption{Distinct eigenpairs for the tensor of Example 3.5 in \cite{koldamayo11siam}. }   
\begin{tabular}{|c|c|c|}            
\hline                        
$\lambda$ & $x$ & stability\\ [0.5ex]   
\hline                              
0.8893 & $\begin{pmatrix}0.67 & 0.25 & -0.70 \end{pmatrix} $ & negatively stable\\              
\hline 
0.8169 & $\begin{pmatrix}0.84 & -0.26 & 0.47 \end{pmatrix}$ & negatively stable\\               
\hline 
0.5105 & $\begin{pmatrix}0.36 & -0.78 & 0.51 \end{pmatrix}$ & unstable\\              
\hline                              
0.3633 & $\begin{pmatrix}0.27 & 0.64 & 0.72 \end{pmatrix}$ & negatively stable\\              
\hline 
0.2682 & $\begin{pmatrix}0.61 & 0.44 & 0.66 \end{pmatrix}$ & unstable\\              
\hline 
0.2628 & $\begin{pmatrix}0.13 & -0.44 & -0.89 \end{pmatrix}$ & unstable\\              
\hline 
0.2433 & $\begin{pmatrix}0.99 & 0.09 & -0.11 \end{pmatrix}$ & unstable\\              
\hline 
0.1735 & $\begin{pmatrix}0.33 & 0.91 & 0.25 \end{pmatrix}$ & unstable\\              
\hline 
-0.0451 & $\begin{pmatrix}0.78 & 0.61 & 0.12 \end{pmatrix}$ & positively stable\\              
\hline 
-0.5629 & $\begin{pmatrix}0.18 & -0.18 & 0.97 \end{pmatrix}$ & positively stable\\              
\hline 
-1.0954 & $\begin{pmatrix}0.59 & -0.75 & -0.30 \end{pmatrix}$ & positively stable\\              
\hline 
\end{tabular}
\label{tbl:roots2}          
\end{center}
\end{table}

Both a positive and negative shift are used for the (adaptive) SS-HOPM runs. Tables \ref{tbl:sshopm2} and \ref{tbl:asshopm2} list the results of applying 200 (100 positive shifts and 100 negative shifts) SS-HOPM and adaptive SS-HOPM runs to this tensor. Using the conservative shift of $37$ and the same convergence criterion as in Example \ref{subsec:Ex1} resulted in 27 runs not converging for the SS-HOPM algorithm. Again, the main strength of the adaptive SS-HOPM is the reduction of number of iterations required to converge. A faster convergence also implies a smaller error as observed in the numerical results. The same initial random vectors were used for both the SS-HOPM and adaptive SS-HOPM runs. Observe that in the 100 adaptive SS-HOPM runs with a negative shift the eigenpair with eigenvalue $-0.0451$ was never found. The total run time for 100 iterations of one adaptive SS-HOPM algorithm took about 0.025 seconds, which means that 100 runs take about 2.5 seconds to complete.

Similarly, the PQRST algorithm without shifting does not converge in any of the permutations. PQRST however retrieves all eigenpairs except one (Table \ref{tbl:pqrst2}). The total time to run PQRST for all permutations for 100 iterations took about 0.35 seconds.
\begin{table}[t!]
\begin{center}
\caption{SS-HOPM for the tensor of Example 3.5 in \cite{koldamayo11siam}. }   
\begin{tabular}{|c|c|c|c|c|}            
\hline                        
occ. & $\lambda$ & $x$ & median its & error \\ [0.5ex]   
\hline                              
2 & 0.8893 & $\begin{pmatrix}0.67 & 0.23 & -0.70 \end{pmatrix} $ & 886 & $2.07\times 10^{-8}$\\               
\hline 
30 & 0.8169 & $\begin{pmatrix}0.84 & -0.26 & 0.47 \end{pmatrix}$  & 644 & $9.15\times 10^{-11}$\\                  
\hline 
41 & 0.3633 & $\begin{pmatrix}0.27 & 0.64 & 0.71 \end{pmatrix}$  & 931 & $4.03\times 10^{-8}$\\                  
\hline                              
55 & -0.0451 & $\begin{pmatrix}-0.78 & -0.61 & -0.12 \end{pmatrix}$  & 680 & $1.81\times 10^{-10}$\\                  
\hline   
27 & -0.5629 & $\begin{pmatrix}-0.18 & 0.18 & -0.97 \end{pmatrix}$  & 386 & $1.37\times 10^{-14}$\\                  
\hline   
18 & -1.0954 & $\begin{pmatrix}-0.59 & 0.75 & 0.30 \end{pmatrix}$  & 357.5 & $1.56\times 10^{-14}$\\                  
\hline   
\end{tabular}
\label{tbl:sshopm2}          
\end{center}
\end{table}

\begin{table}[t!]
\begin{center}
\caption{Adaptive SS-HOPM for the tensor of Example 3.5 in \cite{koldamayo11siam}.}   
\begin{tabular}{|c|c|c|c|c|}            
\hline                        
occ. & $\lambda$ & $x$ & median its & error \\ [0.5ex]   
\hline                              
19 & 0.8893 & $\begin{pmatrix}0.67 & 0.23 & -0.70 \end{pmatrix} $ & 69 & $6.14\times 10^{-16}$\\               
\hline 
30 & 0.8169 & $\begin{pmatrix}0.84 & -0.26 & 0.47 \end{pmatrix}$  & 33 & $8.17\times 10^{-16}$\\                  
\hline 
51 & 0.3633 & $\begin{pmatrix}0.27 & 0.64 & 0.71 \end{pmatrix}$  & 26 & $7.22\times 10^{-16}$\\                  
\hline                              
78 & -0.5629 & $\begin{pmatrix}-0.18 & 0.18 & -0.97 \end{pmatrix}$  & 18 & $7.53\times 10^{-16}$\\                  
\hline   
22 & -1.0954 & $\begin{pmatrix}-0.59 & 0.75 & 0.30 \end{pmatrix}$  & 34 & $6.50\times 10^{-16}$\\                  
\hline                              
\end{tabular}
\label{tbl:asshopm2}          
\end{center}
\end{table}

\begin{table}[t!]
\begin{center}
\caption{PQRST with shift for the tensor of Example 3.5 in \cite{koldamayo11siam}. }   
\begin{tabular}{|c|c|c|c|c|}            
\hline                        
occ. & $\lambda$ & $x$ & median its & error \\ [0.5ex]   
\hline                        
4 & 0.8893 & $\begin{pmatrix}0.67 & 0.23 & -0.70 \end{pmatrix} $ & 96.5 & $1.55\times 10^{-15}$\\               
\hline 
2 & 0.8169 & $\begin{pmatrix}0.84 & -0.26 & 0.47 \end{pmatrix}$  & 70 & $1.37\times 10^{-15}$\\                  
\hline                             
1 & 0.5105 & $\begin{pmatrix}0.36 & -0.78 & 0.51 \end{pmatrix}$  & 33 & $1.34\times 10^{-15}$\\                  
\hline  
4 & 0.2682 & $\begin{pmatrix}0.61 & 0.44 & 0.66 \end{pmatrix}$  & 186.5 & $2.48\times 10^{-15}$\\                  
\hline  
2 & 0.2628 & $\begin{pmatrix}0.13 & -0.44 & -0.87 \end{pmatrix}$  & 46 & $7.80\times 10^{-16}$\\                  
\hline  
2 & 0.2433 & $\begin{pmatrix}0.99 & 0.09 & -0.11 \end{pmatrix}$  & 45.5 & $4.46\times 10^{-16}$\\                  
\hline  
1 & 0.1735 & $\begin{pmatrix}0.33 & 0.91 & 0.25 \end{pmatrix}$  & 45 & $5.85\times 10^{-16}$\\                  
\hline  
6 & -0.0451 & $\begin{pmatrix}-0.78 & -0.61 & -0.12 \end{pmatrix}$  & 34 & $4.20\times 10^{-16}$\\                  
\hline   
2 & -0.5629 & $\begin{pmatrix}-0.18 & 0.18 & -0.97 \end{pmatrix}$  & 19 & $4.60\times 10^{-16}$\\                  
\hline   
4 & -1.0954 & $\begin{pmatrix}-0.59 & 0.75 & 0.30 \end{pmatrix}$  & 25.5 & $1.59\times 10^{-15}$\\                  
\hline 
\end{tabular}
\label{tbl:pqrst2}          
\end{center}
\end{table}

\subsection{Example 3}
\label{subsec:Ex3}
For the third example we consider the 3rd-order tensor from Example 3.6 in \cite{koldamayo11siam}. Solving the polynomial system reveals that there are 4 stable and 3 unstable distinct eigenpairs (Table \ref{tbl:roots3}).

\begin{table}[t!]
\begin{center}
\caption{Distinct eigenpairs for the tensor of Example 3.6 in \cite{koldamayo11siam}. }   
\begin{tabular}{|c|c|c|}            
\hline                        
$\lambda$ & $x$ & stability\\ [0.5ex]   
\hline                              
0.8730 & $\begin{pmatrix}-0.39 & 0.72 & 0.57 \end{pmatrix} $ & negatively stable\\              
\hline 
0.4306 & $\begin{pmatrix}-0.72 & -0.12 & -0.68 \end{pmatrix}$ & positively stable\\               
\hline 
0.2294 & $\begin{pmatrix}-0.84 & 0.44 & -0.31 \end{pmatrix}$ & unstable\\              
\hline                              
0.0180 & $\begin{pmatrix}0.71 & 0.51 & -0.48 \end{pmatrix}$ & negatively stable\\              
\hline 
0.0033 & $\begin{pmatrix}0.45 & 0.77 & -0.45 \end{pmatrix}$ & unstable\\              
\hline 
0.0018 & $\begin{pmatrix}0.33 & 0.63 & -0.70 \end{pmatrix}$ & unstable\\              
\hline 
0.0006 & $\begin{pmatrix}0.29 & 0.73 & -0.61 \end{pmatrix}$ & positively stable\\              
\hline 
\end{tabular}
\label{tbl:roots3}          
\end{center}
\end{table}

The results of running SS-HOPM, adaptive SS-HOPM and PQRST on this tensor exhibit very similar features as in Example \ref{subsec:Ex1}. Just like in Example \ref{subsec:Ex1}, the SS-HOPM algorithm can only retrieve the eigenpair with the largest eigenvalue when the initial random vector is drawn from a uniform distribution. When the initial guess is drawn from a normal distribution, SS-HOPM retrieves 2 eigenpairs in 100 runs while the adaptive SS-HOPM retrieves the 4 stable eigenpairs. (Table \ref{tbl:asshopm3}). The shifted PQRST algorithm ($\delta=.5$) retrieves all 7 eigenpairs (Table \ref{tbl:pqrst3}).

\begin{table}[t!]
\begin{center}
\caption{Adaptive SS-HOPM for the tensor of Example 3.6 in \cite{koldamayo11siam}.}   
\begin{tabular}{|c|c|c|c|c|}            
\hline                        
occ. & $\lambda$ & $x$ & median its & error \\ [0.5ex]   
\hline                              
43 & 0.8730 & $\begin{pmatrix}-0.39 & 0.72 & 0.57 \end{pmatrix} $ & 13 & $4.00\times 10^{-16}$\\               
\hline 
35 & 0.4306 & $\begin{pmatrix}-0.72 & -0.12 & -0.68 \end{pmatrix}$  & 22 & $8.24\times 10^{-16}$\\                  
\hline 
16 & 0.0180 & $\begin{pmatrix}0.71 & 0.51 & -0.48 \end{pmatrix}$  & 43 & $1.26\times 10^{-15}$\\                  
\hline                              
6 & -0.0006 & $\begin{pmatrix}-0.29 & -0.73 & 0.61 \end{pmatrix}$  & 18.5 & $3.35\times 10^{-16}$\\                  
\hline   
\end{tabular}
\label{tbl:asshopm3}          
\end{center}
\end{table}

\begin{table}[t!]
\begin{center}
\caption{PQRST with shift for the tensor of Example 3.6 in \cite{koldamayo11siam}. }   
\begin{tabular}{|c|c|c|c|c|}            
\hline  
occ. & $\lambda$ & $x$ & median its & error \\ [0.5ex]   
\hline                        
1 & 0.8730 & $\begin{pmatrix}-0.39 & 0.72 & 0.57 \end{pmatrix} $  & 569 & $7.09\times 10^{-15}$\\               
\hline 
5 & 0.4306 & $\begin{pmatrix}-0.72 & -0.12 & -0.68 \end{pmatrix}$  & 56 & $1.79\times 10^{-15}$\\                  
\hline 
2 & 0.2294 & $\begin{pmatrix}-0.84 & 0.44 & -0.31 \end{pmatrix}$ & 511 & $1.65\times 10^{-15}$\\                     
\hline                              
2 & 0.0180 & $\begin{pmatrix}0.71 & 0.51 & -0.48 \end{pmatrix}$ & 873 & $5.58\times 10^{-16}$\\                           
\hline 
2 & 0.0033 & $\begin{pmatrix}0.45 & 0.77 & -0.45 \end{pmatrix}$ & 1328 & $4.03\times 10^{-16}$\\                             
\hline 
2 & 0.0018 & $\begin{pmatrix}0.33 & 0.63 & -0.70 \end{pmatrix}$ & 1177 & $4.76\times 10^{-16}$\\                              
\hline 
2 & 0.0006 & $\begin{pmatrix}0.29 & 0.73 & -0.61 \end{pmatrix}$ & 3929.5 & $3.63\times 10^{-16}$\\                   
\hline 
\end{tabular}
\label{tbl:pqrst3}          
\end{center}
\end{table}

\subsection{Example 4}
\label{subsec:Ex4}
As a last example we consider a random symmetric tensor in $\mathbb{S}^{[3,6]}$. Both SS-HOPM and the adaptive SS-HOPM are able to find 4 distinct eigenpairs in 100 runs. Since the difference in the results between SS-HOPM and the adaptive SS-HOPM lies in the total number of iterations we only list the results for the adaptive SS-HOPM method in Table \ref{tbl:asshopm4}.

\begin{table}[t!]
\begin{center}
\caption{Adaptive SS-HOPM for random symmetric tensor of order 3 and dimension 6.}   
\begin{tabular}{|c|c|c|c|c|}            
\hline                        
occ. & $\lambda$ & $x$ & median its & error \\ [0.5ex]   
\hline                              
23 & 2.4333 & $\begin{pmatrix}-0.64 & -0.36 & 0.57 & 0.33 & - 0.17 & -0.07 \end{pmatrix} $ & 79 & $1.87\times 10^{-14}$\\               
\hline 
23 & 4.3544 & $\begin{pmatrix}-0.07 & 0.30 & -0.38 & 0.48 & 0.68 & 0.26 \end{pmatrix}$  & 25 & $5.96\times 10^{-15}$\\                  
\hline 
33 & 3.4441 & $\begin{pmatrix}0.23 & -0.16 & 0.38 & -0.72 & 0.43 & 0.27 \end{pmatrix}$  & 52 & $1.02\times 10^{-14}$\\                  
\hline                              
21 & 2.8488 & $\begin{pmatrix}-0.02 & -0.26 & -0.30 & 0.31 & -0.52 & 0.69 \end{pmatrix}$  & 70 & $9.45\times 10^{-15}$\\                  
\hline   
\end{tabular}
\label{tbl:asshopm4}          
\end{center}
\end{table}

In contrast, PQRST with shift $\delta=1$ finds 7 distinct eigenpairs. The 3 additional eigenpairs are unstable.

\begin{table}[t!]
\begin{center}
\caption{PQRST for random symmetric tensor of order 3 and dimension 6.}   
\begin{tabular}{|c|c|c|c|c|}            
\hline                        
occ. & $\lambda$ & $x$ & median its & error \\ [0.5ex]   
\hline
24 & 1.3030 & $\begin{pmatrix}-0.47 & 0.84 & 0.03 & -0.07 &  0.05 & 0.27 \end{pmatrix} $ & 599 & $6.78\times 10^{-14}$\\
\hline                              
576 & 2.4333 & $\begin{pmatrix}-0.64 & -0.36 & 0.57 & 0.33 & - 0.17 & -0.07 \end{pmatrix} $ & 75 & $5.68\times 10^{-11}$\\               
\hline 
216 & 4.3544 & $\begin{pmatrix}-0.07 & 0.30 & -0.38 & 0.48 & 0.68 & 0.26 \end{pmatrix}$  & 148 & $2.66\times 10^{-13}$\\                  
\hline 
120 & 3.4441 & $\begin{pmatrix}0.23 & -0.16 & 0.38 & -0.72 & 0.43 & 0.27 \end{pmatrix}$  & 41 & $6.90\times 10^{-14}$\\                  
\hline                              
240 & 2.8488 & $\begin{pmatrix}-0.02 & -0.26 & -0.30 & 0.31 & -0.52 & 0.69 \end{pmatrix}$  & 48 & $6.63\times 10^{-14}$\\                  
\hline   
120 & 1.6652 & $\begin{pmatrix}-0.27 & 0.37 & 0.02 & -0.31 & -0.26 & 0.79 \end{pmatrix}$  & 773 & $2.66\times 10^{-13}$\\                  
\hline   
132 & 1.2643 & $\begin{pmatrix}-0.02 & -0.51 & -0.48 & -0.61 & 0.22 & -0.30 \end{pmatrix}$  & 150 & $2.65\times 10^{-13}$\\                  
\hline
\end{tabular}
\label{tbl:pqrst4}          
\end{center}
\end{table}

\section{Conclusions}
\label{sec:conclusions}
The QR algorithm for finding eigenpairs of a matrix has been extended to its symmetric tensor counterpart called QRST (QR algorithm for symmetric tensors). The QRST algorithm, like its matrix version, permits a shift to accelerate convergence. A preconditioned QRST algorithm, called PQRST, has further been proposed to efficiently locate multiple eigenpairs of a symmetric tensor by orienting the tensor into various directions. Numerical examples have verified the effectiveness of (P)QRST in locating stable and unstable eigenpairs not found by previous tensor power methods.

\bibliographystyle{siam}
\bibliography{references}

\begin{thebibliography}{10}

\bibitem{Comon2006b}
{\sc P.~Comon and M.~Rajih}, {\em Blind identification of under-determined
  mixtures based on the characteristic function}, Signal Processing, 86 (2006),
  pp.~2271 -- 2281.
\newblock Special Section: Signal Processing in \{UWB\} Communications.

\bibitem{Lathauwer2007}
{\sc L.~De~Lathauwer, J.~Castaing, and J.~Cardoso}, {\em Fourth-order
  cumulant-based blind identification of underdetermined mixtures}, Signal
  Processing, IEEE Transactions on, 55 (2007), pp.~2965--2973.

\bibitem{de2000best}
{\sc L.~De~Lathauwer, B.~De~Moor, and J.~Vandewalle}, {\em On the best rank-$1$
  and rank-{($R_1$, $R_2$,..., $R_n$)} approximation of higher-order tensors},
  SIAM Journal on Matrix Analysis and Applications, 21 (2000), pp.~1324--1342.

\bibitem{PhilippeDreesenThesis}
{\sc P.~Dreesen}, {\em {Back to the Roots: Polynomial System Solving Using
  Linear Algebra}}, PhD thesis, Faculty of Engineering, KU Leuven (Leuven,
  Belgium), 2013.

\bibitem{Ferreol2005}
{\sc A~Ferreol, L.~Albera, and P.~Chevalier}, {\em Fourth-order blind
  identification of underdetermined mixtures of sources ({FOBIUM})}, Signal
  Processing, IEEE Transactions on, 53 (2005), pp.~1640--1653.

\bibitem{matrixcomputations}
{\sc G.~H. Golub and C.~F. Van~Loan}, {\em Matrix Computations}, The Johns
  Hopkins University Press, 3rd~ed., Oct. 1996.

\bibitem{kilmer2004decomposing}
{\sc M.~E Kilmer and C.~D.~M. Martin}, {\em Decomposing a tensor}, SIAM News,
  37 (2004), pp.~19--20.

\bibitem{kofidis2002best}
{\sc E.~Kofidis and P.~A. Regalia}, {\em On the best rank-1 approximation of
  higher-order supersymmetric tensors}, SIAM Journal on Matrix Analysis and
  Applications, 23 (2002), pp.~863--884.

\bibitem{tensorreview}
{\sc T.~Kolda and B.~Bader}, {\em Tensor decompositions and applications}, SIAM
  Review, 51 (2009), pp.~455--500.

\bibitem{koldamayo11siam}
{\sc T.~G. Kolda and J.~R. Mayo}, {\em Shifted power method for computing
  tensor eigenpairs}, SIAM Journal on Matrix Analysis and Applications, 32
  (2011), pp.~1095--1124.

\bibitem{koldamayo14arxiv}
{\sc T.~G. {Kolda} and J.~R. {Mayo}}, {\em An adaptive shifted power method for
  computing generalized tensor eigenpairs}, ArXiv e-prints,  (2014).

\bibitem{koldamayo14asshopm}
\leavevmode\vrule height 2pt depth -1.6pt width 23pt, {\em {An Adaptive Shifted
  Power Method for Computing Generalized Tensor Eigenpairs}}, ArXiv e-prints,
  (2014).

\bibitem{LandsbergBook2012}
{\sc J.~M. Landsberg}, {\em Tensors: Geometry and Applications}, American
  Mathematical Society, 2012.

\bibitem{Lim_singularvalues}
{\sc L.-H. Lim}, {\em Singular values and eigenvalues of tensors: a variational
  approach}, in in CAMAP2005: 1st IEEE International Workshop on Computational
  Advances in Multi-Sensor Adaptive Processing, 2005, pp.~129--132.

\bibitem{lim2014blind}
{\sc L.-H. Lim and P.~Comon}, {\em Blind multilinear identification}, IEEE
  Trans. Inf. Theory, 60 (2014), pp.~1260--1280.

\bibitem{qi05jsc}
{\sc L.~Qi}, {\em Eigenvalues of a real supersymmetric tensor}, Journal of
  Symbolic Computation, 40 (2005), pp.~1302--1324.

\end{thebibliography}

\end{document}